\documentclass[12pt]{amsart}
\usepackage{amsmath,amsthm,amsfonts,amssymb,times,latexsym,mathabx,url}
\usepackage[pdftex]{graphicx}

\newtheorem{theorem}{Theorem}[section]

\newtheorem{lem}[theorem]{Lemma}
\newtheorem{cor}[theorem]{Corollary}

\numberwithin{equation}{section}

\newcommand{\e}{\varepsilon}
\renewcommand{\o}{\omega}
\renewcommand{\O}{\Omega}
\renewcommand{\a}{\alpha}

\renewcommand{\P}{\mathbb{P}}
\newcommand{\N}{\mathbb{N}}
\newcommand{\Z}{\mathbb{Z}}
\renewcommand{\leq}{\leqslant}
\renewcommand{\geq}{\geqslant}

\newcommand{\bfZ}{\mathbf Z}

\newcommand{\sumprime}{\sideset{}{'}\sum}
\newcommand{\sssum}{\sideset{}{''}\sum}

\newcommand{\NN}{{\mathbb N}}

\newcommand{\PR}{\mathbb{P}}
\newcommand{\er}{\ensuremath{\mathrm{e}}}

\makeatletter
\renewcommand{\pmod}[1]{\allowbreak\mkern7mu({\operator@font mod}\,\,#1)}
\makeatother

\newcommand{\be}{\begin{equation}}
\newcommand{\ee}{\end{equation}}

\renewcommand{\a}{\ensuremath{\alpha}}

\newcommand{\eps}{\ensuremath{\varepsilon}}

\renewcommand{\le}{\leqslant}
\renewcommand{\leq}{\leqslant}
\renewcommand{\ge}{\geqslant}
\renewcommand{\geq}{\geqslant}

\newcommand{\pfrac}[2]{\left(\frac{#1}{#2}\right)}  %%% frac with paren

\renewcommand{\ssum}[1]{\sum_{\substack{#1}}}  %%% stacked sum
\begin{document}

\title{Sets whose differences avoid squares modulo $m$}
\author{Kevin Ford, Mikhail R. Gabdullin}
\date{New estimates for sets whose differece avoids squares}
\address{Department of Mathematics, 1409 West Green Street, University
of Illinois at Urbana-Champaign, Urbana, IL 61801, USA}
\email{ford@math.uiuc.edu}
\address{
Steklov Mathematical Institute,
Gubkina str., 8, Moscow, Russia, 119991}
\email{gabdullin.mikhail@yandex.ru, gabdullin@mi-ras.ru} 

\begin{abstract}
	We prove that if $\e(m)\to0$ arbitrarily slowly, then for almost all $m$ and any $A\subset\Z_m$ such that $A-A$ does not contain non-zero quadratic residues we have
	$|A|\leq m^{1/2-\e(m)}.$
\end{abstract}

\date{\today}

\maketitle

%%%%%%%%%%%%%%%%%%%%%%%%%%%%%%%%%%%%%%%%%%%%%%%%%%%%%%%%%%%%%%%%

\section{Introduction}

Let $R_m=\{a^2: a\in\Z_m \}$ be the set of quadratic residues modulo $m$. In this paper we find an upper bound for the sets $A\subset\Z_m$ with
\begin{equation}\label{0.1}
(A-A)\cap R_m = \{0 \},
\end{equation}
where $A-A=\{a-b: a,b\in A \}$, for a large set of $m$. This question originated from the corresponding problem in $\Z$\,: Ruzsa \cite{R} constructed a set of integers $B\subset[1,N]$ such that $B-B$ avoids squares with $|B|\gg N^{\gamma}$, where $\gamma=\frac12(1+\frac{\log 7}{\log 65})=0.73...$, and this construction was based on a $7$-element subset of $\Z_{65}$ obeying (\ref{0.1}) (see \cite{Le} for further improvements). As for upper bounds in this integer setting, we just note that such a set $B$ must obey $|B|=o(N)$ (see \cite{Sa}, \cite{PSS}, and also \cite{BPPS}), but no bounds with power saving are known.

Now we begin a discussion of (\ref{0.1}). In the case of a prime $m\equiv3\pmod{4}$, $(\frac{-1}{m})=-1$ and thus any set $A\subset\Z_m$ with (\ref{0.1}) is a singleton or empty at all, and so the problem is trivial. In contrast, in the case of a prime $m=p\equiv1\pmod{4}$ it should be very hard to obtain good bounds, since the problem is related to two other famous questions. The first one is the clique number problem for the Paley graph. Recall that the Paley graph is the graph $G_p=(V,E)$ with $V=\Z_p$ and $\{a,b\}\in E$ iff $a-b$ is a quadratic residue modulo $p$, and a clique of an undirected graph is a subset of its vertices such that every two distinct vertices in the clique are adjacent (that is, its induced subgraph is complete). The clique number of a graph is the size of its maximum clique. Fix any quadratic non-residue $\xi\in\Z_p\setminus R_p$; then $C\subset\Z_p$ is a clique of $G_p$ if and only if $\xi C$ obeys (\ref{0.1}), and so any bound for the clique number is a bound for our sets, and vice versa. While it is not hard to show that any clique in this graph (and, hence, any set $A\subset\Z_p$ with (\ref{0.1})) has size at most $p^{1/2}$ (see Section \ref{Sec2} for a short proof), any improvement of it requires non-trivial ideas; currently the best upper bound is $\sqrt{p/2}+1$ (see \cite{HP}). The second problem related to our sets is finding an upper bound for the least quadratic non-residue. If we denote it by $n(p)$, then the set $\xi\cdot\{1,\ldots,n(p) \}$, where $\xi\in\Z_p\setminus R_p$ is again any quadratic non-residue, has size $n(p)$ and obeys (\ref{0.1}), and so any bound for sets with (\ref{0.1}) is a bound for $n(p)$, 
but the best we know is $n(p)\ll p^{1/4\sqrt{e}+o(1)}$, due to the work \cite{Bu} (see also the classical papers \cite{V} and \cite{Li}).   

We turn to the case of composite $m$. Matolcsi and Ruzsa \cite{MR} proved that $|A|\leq m^{1/2}$ for all $A\subset\Z_m$ with (\ref{0.1}) and square-free $m$ which have prime divisors of the form $4k+1$ only, and that $|A|\leq m\exp(-c\sqrt{\log m})$ for all square-free $m$ (throughout the paper we denote by $c$ absolute positive constants which may vary from line to line). The second author \cite{G} proved that for any square-free $m$ and such $A$ we have 
$$|A|\leq \min \left\{ m^{1/2}(3\o(m))^{3\o(m)/2}, \, m\exp\left(-\frac{c\log m}{\log\log m}\right)\right\},
$$
where $\o(m)$ is the number of prime divisors of $m$.
Since $\o(m)\le 2\log\log m$ for almost all $m$ (due to Hardy and Ramanujan \cite{HR})
we deduce  that $|A|\leq m^{1/2+o(1)}$ for almost all $m$ (that is, for a set of density $1$; here and in what follows we consider the lower asymptotic density of a set $M\subseteq \N$, which is defined as $\varliminf_{N\to\infty}\frac{\#(M\cap\{1,...N\})}{N}$). 

\smallskip 

In this paper we overcome this square-root barrier for almost all moduli. 
For a positive integer $m$, we denote by $\o_3(m)$ the number of its prime divisors of the form $4k+3$. Our main results are the following.

\begin{theorem}\label{th1.1} 
	Let $m$ be square-free and let $A\subset\Z_m$ obey (\ref{0.1}). Then  
	$$|A|\leq m^{1/2}q^{-1/2}(10\o(m))^{2\o(m)},
	$$	
	where $q$ denotes the least prime divisor of $m$ of the form $4k+3$ if $\o_3(m)$ is odd, and $q=1$ otherwise. 	
\end{theorem}

It is an improvement (due to the factor $q^{-1/2}$) of the mentioned result of \cite{G} (note that we obtain worse constants $10$ and $2$ instead of $3$ and $3/2$, but it is not so important). While it is not useful directly for all moduli, it allows us (using some ``truncation'' trick) to obtain a bound $o(m^{1/2})$ for almost all $m$. 
 
\begin{theorem}\label{th1.2}
			Let $\e\in[(\log x)^{-1/2},1]$ and $c(\e)=\exp(-(\log\e^{-1})^{1/10})$. Then for all but $O(c(\e)x)$  numbers $m\leq x$ and any $A\subset\Z_m$ with (\ref{0.1}) we have 
$$|A|\leq m^{1/2-\e/5}.
$$
\end{theorem}

We immediately conclude the following.

\begin{cor}\label{cor1.3}
Let $\e(m)\to0$ arbitrarily slowly. Then 
$$|A|\leq m^{1/2-\e(m)}
$$
for almost all $m$ and $A\subset\Z_m$ with (\ref{0.1}).
\end{cor}

We now discuss possible improvements of this bound. For $\eta\in(0,1)$, we set
$$
M_{\eta}=\{m\in\N: \mbox{ for any $A\subset\Z_m$ with (\ref{0.1})  the bound $|A|\leq m^{\eta}$ holds }  \} .
$$
Using this notation, we can reformulate the mentioned corollary from \cite{G} as follows: for any $\e>0$ the set $M_{1/2+\e}$ has density $1$. Theorem \ref{th1.2} can also be presented in terms of these sets $M_{\eta}$: it means that the density of the set $M_{1/2-\e}$ tends to $1$ as $\e\to0$.

Note that in the case $m=p^2$, where $p$ is a prime, $\Z_{p^2}$ contains the set $\{0,p,2p,...,(p-1)p\}$ which has size $m^{1/2}$ and obeys (\ref{0.1}) (see also Proposition 5.1 of \cite{Y} for a more general statement). Nevertheless, it is believed that for any square-free $m$ and $A\subset\Z_m$  with (\ref{0.1}) the bound $|A|\ll_{\e}m^{\e}$ holds for any $\e>0$, and, hence, that the set $M_{\e}$ has density $1$. While this hypothesis seems to be far beyond the reach of currect methods, one can prove the following weak form of it.

\begin{theorem}\label{th1.4}
For any $\e>0$ the set $M_{\e}$ has positive density. 	
\end{theorem}

Finally, we mention a lower bound for almost all moduli.

\begin{theorem}\label{th1.5}
	For almost all $m$ there exists a set $A\subset\Z_m$ with \\  $(A-A)\cap R_m=\{0\}$ and
	$$|A|\geq \exp(0.375(\log\log m)^2(1+o(1))).
	$$
\end{theorem}

\bigskip

In Section \ref{Sec2} we prove Theorem \ref{th1.1}; we closely follow the proof of the main result of \cite{G} with some modifications. In Section \ref{Sec3} we use Theorem \ref{th1.1} and some ``truncation'' argument to reduce Theorem \ref{th1.2} to Lemma \ref{lem3.2}, which concernes with the properties of large prime divisors of the form $4k+3$ of typical integers. In Section \ref{Sec4} we prove Lemma \ref{lem3.2} and thus finish the proof of Theorem \ref{th1.2}. The proof of Lemma \ref{lem3.2} relies on the fact that if $T_1,...,T_r$ are disjoint subset of primes in the interval $[y,z]\subset[2,x]$, where $y\to\infty$ and $\log x/\log z\to\infty$, then $\o(n,T_j)$ behave like independent Poisson random variables with parameters $H(T_j)=\sum_{p\in T_j}p^{-1}$. Section \ref{Sec5} is devoted to Theorems \ref{th1.4} and \ref{th1.5}. The proof of Theorem \ref{th1.4} relies on the observation that we have the bound $|A|\leq m^{\e}$ for any $A\subset\Z_m$ with (\ref{0.1}) whenever $m$ has a prime divisor $q\equiv3\pmod{4}$ such that $q\geq m^{1-\e}$. Finally, Theorem \ref{th1.5} is obtained by a ``product'' argument from the corresponding lower bounds for the cases of a prime $m=p\equiv1\pmod{4}$ and $m=q_1q_2$, where $q_1$ and $q_2$ are primes $3\pmod{4}$.

\section{Proof of Theorem \ref{th1.1}}
\label{Sec2}

In what follows, we use the words ``residue''and ``non-residue'' for ``quadratic residue'' and ``quadratic non-residue'' respectively.

Firstly, we show that it is enough to prove the theorem for odd $m$. Suppose that $m$ is even and write $m=2m_1$; then $\Z_m=\Z_2\oplus\Z_{m_1}$. Set
$$A_1=A\pmod{m_1}=\{ x\in\Z_{m_1} : \mbox{ there exists } a\in\Z_2 \mbox{ with } (a,x)\in A \}.
$$
Note that for any $x\in\Z_{m_1}$ at most one of the elements $(0,x)$ and $(1,x)$ belongs to $A$ (since their difference $(1,0)$ is a residue modulo $m$); denote this element, if it exists, by $(a_x,x)$. Hence, $|A|=|A_1|$.  Further, for any distinct $x,y\in A_1$ the difference $x-y$ is a non-residue modulo $m_1$ (since otherwise the difference $(a_x-a_y, x-y)$ would be a non-zero residue modulo $m$), and so without loss of generality we may assume that $m$ is odd.

Now we prove the theorem for odd $m$. We induct on $n=\o(m)$. Let $n=1$, that is, $m=p$ is a prime. If $p\equiv3\pmod{4}$, then $|A|\leq1$, since $-1$ is a non-residue modulo $p$ and, hence, for any $a\neq b$ one of the differences $a-b$ or $b-a$ is a residue modulo $p$. If $p\equiv1\pmod{4}$, we have the bound $|A|\leq p^{1/2}$. We give an elegant and folklore proof for that. Let us assume that $|A|>p^{1/2}$ and fix a non-residue $\xi\in\Z_p$. Consider the map $\varphi\colon A^2\to\Z_p$ defined by $\varphi(a,b)=a+\xi b$. By the pigeonhole principle, there are two distinct pairs $(a_1,b_1)$ and $(a_2,b_2)$ with $\varphi(a_1,b_1)=\varphi(a_2,b_2)$, that is, $\xi=(a_1-a_2)(b_2-b_1)^{-1}$. It follows that one of the differences $a_1-a_2$ and $b_1-b_2$ is a nonresidue modulo $p$, and we are done.

Now assume that $n\geq2$ and the claim is true for all square-free $m'$ with $\o(m')<n$. 
Let $p_1<p_2<\ldots<p_n$ be the prime divisors of $m$. Denote by $\chi_j$ the Legendre symbol modulo $p_j$. Since each difference $a-b$ of distinct elements of $A$ is a non-residue modulo $p_j$ for at least one $p_j$, we have
\begin{equation}\label{2.1} 
|A|=\sum_{a,b\in A}\prod_{j=1}^n (1+\chi_j(a-b))
= |A|^2 + \sum_{D}\sum_{a,b\in A}\chi_D(a-b),
\end{equation}
where $D$ runs over all non-empty subsets of $[n]=\{1,\ldots,n\}$ and $\chi_D(x)=\prod_{j\in D}\chi_j(x)$. Set $p_D=\prod_{j\in D}p_j$. The key observation which makes possible our improvement of the main result of \cite{G} is that we may restrict the outer summation over those (non-empty) $D$ for which $\o_3(p_D)$ is even (since otherwise $\chi_D(-1)=-1$ and $\chi_D(a-b)+\chi_D(b-a)=0$ for any $a,b\in\Z_{p_D}$). In what follows, we denote the summation over such $D$ by $\sumprime_D$.

Denote $\sigma=1-|A|^{-1}$. Then we may rewrite (\ref{2.1}) as follows:
\begin{equation*}
|A|^2\sigma=-\sumprime_{D}\sum_{a,b\in A} \chi_D(a-b).
\end{equation*}
Using Cauchy-Schwarz, we see that
\begin{equation*} |A|^2\sigma\leq \sumprime_{D}|A|^{1/2}S_{D}^{1/2}, 
\end{equation*}
where
$$ S_{D}=\sum_{a\in A} \left|\sum_{b \in A} \chi_D(a-b)\right|^{2}.
$$
Thus
\begin{equation} \label{A}
|A|^{3/2}\sigma\leq \sumprime_{D}S_{D}^{1/2}.
\end{equation}
Now we need to estimate the sums $S_D$. For $D\subseteq[n]$ we set
$$
H_D=\max \{ |A|: A\subset \Z_{mp_D^{-1}} , \, A \mbox{ obeys } (\ref{0.1})  \} .
$$

The following bound is crucial for the induction step.

\begin{lem}\label{lem2.2} 
For any non-empty $D\subseteq[n]$ we have
$$ S_D \leq |A|^2H_D+|A|H_D\sum_{D'\subseteq D} H_{D'}p_{D'}
$$	
(here and in what follows the summation is over non-empty $D'$).
\end{lem}

\begin{proof}  For each residue $x$ modulo $p_D$ we set 
$$A_x=\{a\in A : a\equiv x \pmod{p_D}\}.
$$
One can think of elements of $A_x$ modulo $mp_D^{-1}$, and the difference of distinct elements of $A_x$ is a non-residue modulo $mp_D^{-1}$. Then by the definition of $H_D$ we have $|A_x|\leq H_D$;
further, obviously, $A=\bigsqcup\limits_{x\in\mathbb{Z}_{p_D}} A_x$ and elements of $A_x$ give the same contribution to $S_D$. We thus see that
\begin{align*}
S_D &= \sum_{x\in \Z_{p_D}} \sum_{a\in A_x}\left|\sum_{b \in A} \chi_D(x-b)\right|^{2} =\sum_{x\in \Z_{p_D}} |A_x|\left|\sum_{b \in A} \chi_D(x-b)\right|^{2}  \\
&\le H_D\sum_{b_1,b_2\in A}\sum_{x\in\Z_{p_D}}\prod_{j\in D} \chi_j(x-b_1)\chi_j(x-b_2)  \\
&=H_D\sum_{b_1,b_2\in A}\prod_{j\in D}\sum_{x_j\in\Z_{p_j}} \chi_j(x_j-b_1)\chi_j(x_j-b_2). 
\end{align*}
Let us compute the inner sum. 
For the sake of brevity we introduce the following definition: a pair $(b_1,b_2)$ is said to be \textit{special} modulo $p$ if $b_1\equiv b_2 \pmod{p}$. We have 
\[
\sum_{x\in \Z_{p_j}}\chi_j(x-b_1)\chi_j(x-b_2)=p_j-1
\]
 if $(b_1,b_2)$ is special modulo $p_j$, and 
$$\sum_{x\in \Z_{p_j}}\chi_j(x-b_1)\chi_j(x-b_2)=\sum_{x\neq b_2} \chi_j\left(1+\frac{b_2-b_1}{x-b_2}\right)=\ssum{x\in \Z_{p_j} \\ x\neq1} \chi_j(x)=-1
$$ 
otherwise.

Denote by $B_{D'}$ the set of pairs $(b_1,b_2)\in A^2$ which are special modulo each prime $p_j$ with $j\in D'$, and not special
modulo every prime $p_j$ with $j\in D \setminus D'$.
In particular, $(b_1,b_2)\in B_{D'}$ implies that
$b_1\equiv b_2 \pmod{p_{D'}}$ and thus
$|B_{D'}| \le |A| H_{D'}$.
 We thus have
 \begin{align*}
 S_D& \leq H_D\left( (-1)^{|D|}|B_{\varnothing}| + \sum_{D'\subseteq D} (-1)^{|D|-|D'|} \phi(p_{D'})|B_{D'}| \right)\\
&\le H_D |A|^2 + |A| H_D \sum_{D'\subseteq D} p_{D'} H_{D'},
 \end{align*}
 where $\phi$ is Euler's function. 
\end{proof}

\bigskip

This lemma implies
$$S_D^{1/2} \leq |A|H_D^{1/2} + |A|^{1/2}H_D^{1/2}\sum_{D'\subseteq D}H_{D'}^{1/2}p_{D'}^{1/2} . 
$$
Substituting this estimate into (\ref{A}), we obtain
\begin{equation}
|A|\sigma \leq |A|^{1/2}T_1 + T_2, \label{6} 
\end{equation}
where
\begin{equation*} T_1=\sumprime_{D\subseteq[n]}H_D^{1/2}, \label{T_1}
\end{equation*} 
\begin{equation*}T_2=\sumprime_{D\subseteq[n]} H_D^{1/2}\sum_{D'\subseteq D}H_{D'}^{1/2}p_{D'}^{1/2}. \label{T_2}
\end{equation*}

%\bigskip\bigskip\bigskip

Our further argument is roughly the following. By the induction hypothesis we have $H_D\ll_n (mp_D^{-1})^{1/2}$ (as is usual, the notation $B\ll_n C$ with positive $B,C$ means that $B\leq f(n)C$ for an appropriate function $f$); then
$$T_1 \ll_n m^{1/4}\sum_Dp_D^{-1/4} \ll m^{1/4}\sum_D1 \ll_n m^{1/4}
$$
and, similarly,
$$T_2 \ll_n m^{1/2}\sum_D p_D^{-1/4}\sum_{D'\subseteq D}p_{D'}^{1/4} \ll m^{1/2} \sum_D\sum_{D'\subseteq D}1 \ll_n m^{1/2}.
$$
Hence, (\ref{6}) gives us
$$|A|\ll_n |A|^{1/2}m^{1/4} + m^{1/2}, 
$$
and we get a contradiction if $|A|\gg_nm^{1/2}$. So we easily have a bound $|A|\ll_nm^{1/2}$ with some explicit dependence of the constant on $n$, and this is enough to prove the theorem in the case where $\o_3(m)$ is even. If $\o_3(m)$ is odd, we have a better bound $H_D\ll_n (mp_D^{-1}q^{-1})^{1/2}$ if $\o_3(p_D)$ is even, and the similar argument gives $T_1\ll_n (mq^{-1})^{1/4}$; the only problem for getting immediately the bound $T_2\ll_n (mq^{-1})^{1/2}$ (which would imply the theorem in the same way up to explicit dependence on $n$) is that we have only ``trivial'' bound $H_{D'}\ll_n (mp_{D'}^{-1})^{1/2}$ if $\o_3(p_{D'})$ is odd. However, it turns out to be 
just a technical difficulty and we are still able to proceed as above. Note that these crude bounds for $T_1$ and $T_2$ imply the theorem in the form $|A|\leq m^{1/2}q^{-1/2}f(n)$ with $f(n)=\exp(O(n^2))$, whereas we are aiming for a better dependence on $n$.

We turn to the details. For $D\subseteq[n]$, $D\neq\varnothing$, let $q_D$ be the least prime divisor of $mp_D^{-1}$ of the form $4k+3$ if $\o_3(mp_D^{-1})$ is odd, and $q_D=1$ otherwise. Since $D$ is non-empty, we have $\o(mp_D^{-1})<n$ and we can apply the induction hypothesis, which gives us 
\begin{equation}\label{H_Dtr}
H_D\leq (mp_D^{-1})^{1/2}q_D^{-1/2}(10n)^{2(n-|D|)}. 
\end{equation}
Recall that the summation in the sums $T_1$ and $T_2$ is taken over $D$ with even $\o_3(p_D)$ (let us call these $D$ \textit{proper}). Set $q=q_{\varnothing}$; if $\o_3(m)$ is even, then $q_D=q=1$ for any proper $D$. If $\o_3(m)$ is odd, then $q_D\geq q$ for any proper $D$. Hence, in both cases we have
\begin{equation}
H_D\leq (mp_D^{-1})^{1/2}q^{-1/2}(10n)^{2(n-|D|)}. \label{H_D}
\end{equation}
for any proper $D$.

%\if 0  text: blabla  \fi 

Now we estimate the sums $T_1$ and $T_2$. We begin with a bound for $T_1$.
% Since $p_1\geq3$ and the function $u^{-1/4}$ is concave, we have
%\begin{multline}
%\sum_{j=1}^n p_j^{-1/4}\leq\sum_{j=1}^n(2j+1)^{-1/4}\leq 0.5\sum_{j=1}^n\int_{2j}^{2j+2}u^{-1/4}du=\\
%\frac23((2n+2)^{3/4}-2^{3/4}) < \frac23(2n)^{3/4} < 1.13n^{3/4}.  \label{1.13}
%\end{multline}
Using (\ref{H_D}) and extending the summation to all $D$, we have
\begin{multline}
(mq^{-1})^{-1/4}\,T_1\leq \sumprime_{D}p_D^{-1/4}(10n)^{n-|D|} \leq  (10n)^n\sum_{d=1}^n\sum_{|D|=d}p_D^{-1/4}(10n)^{-d}  \leq \\ (10n)^n\sum_{d=1}^n(10n)^{-d}\binom{n}{d}\leq (10n)^n\sum_{d=1}^n(10n)^{-d}\frac{n^d} {d!}\leq\\
(10n)^{n}\sum_{d=1}^{\infty}\frac{10^{-d}}{d!}\leq  0.12(10n)^{n}. \label{T1} 
\end{multline}
Finally, we estimate $T_2$. We have
\begin{equation}\label{prnpr} 
T_2=T_2'+T_2'',
\end{equation}
where
\begin{equation*}
T_2'=\sumprime_{D} H_D^{1/2}\sumprime_{D'\subseteq D}H_{D'}^{1/2}p_{D'}^{1/2}
\end{equation*}
(the inner summation is over proper $D'$), and
\begin{equation*}
T_2''=\sumprime_{D} H_D^{1/2}\sssum_{D'\subseteq D}H_{D'}^{1/2}p_{D'}^{1/2}
\end{equation*}
(the inner summation is over non-proper $D'$). We first work with $T_2'$. Using the bound (\ref{H_D}), we find 
$$H_{D'}p_{D'}\leq (mp_{D'}^{-1})^{1/2}q^{-1/2}(10n)^{2(n-|D'|)}p_{D'}=m^{1/2}p_{D'}^{1/2}q^{-1/2}(10n)^{2(n-|D'|)},
$$
and hence
\begin{equation}\label{T_2'}
T_2' \leq (mq^{-1})^{1/2}\sumprime_{D}(10n)^{n-|D|}p_D^{-1/4}\sumprime_{D'\subseteq D}(10n)^{n-|D'|}p_{D'}^{1/4}.
\end{equation}
Now we estimate $T_2''$. Fix some proper $D$ and non-proper $D'\subseteq D$. Since $\o_3(p_D)$ is even and $\o_3(p_{D'})$ is odd, $D'$ can be represented in the form $D'=D_1\setminus\{p_j\}$ with proper $D_1\subseteq D$ such that $\o_3({p_{D_1}})\geq2$ and some $p_j\equiv3\pmod{4}$; note that $p_j\geq q$. Thus, using the bounds (\ref{H_Dtr}) and $q_{D'}\geq 1$, we have
$$H_{D'}p_{D'}\leq (mp_{D_1}^{-1}p_j)^{1/2}(10n)^{2(n-|D_1|+1)}p_{D_1}p_j^{-1}\leq m^{1/2}p_{D_1}^{1/2}q^{-1/2}(10n)^{2(n-|D_1|+1)}.
$$
Then we may rewrite the inner summation in $T_2''$ over non-proper $D'\subseteq D$ as the summation over proper $D_1$ with $\o_3(p_{D_1})\geq2$; any bound of the above type occurs at most $|D_1|$ times, and hence by \eqref{H_D} we have 
\be\label{T_2''}
T_2'' \leq (mq^{-1})^{1/2}\sumprime_{D}(10n)^{n-|D|}p_D^{-1/4}\sumprime_{\substack{D_1\subseteq D \\ \o_3(p_{D_1})\geq2} }|D_1|(10n)^{n-|D_1|+1}p_{D_1}^{1/4}.
\ee
Combining (\ref{T_2'}) and (\ref{T_2''}) with (\ref{prnpr}), we obtain
\begin{multline*}
(mq^{-1})^{-1/2}(10n)^{-2n}T_2 \leq \\ \sumprime_{D}(10n)^{-|D|}\sumprime_{D_1\subseteq D}(|D_1|+1)(10n)^{-(|D_1|-1)}(p_{D_1}/p_D)^{1/4}\leq \\
\sum_{D}(10n)^{-|D|}\sum_{D_1\subseteq D}(|D_1|+1)\leq 
\sum_{D}(10n)^{-|D|}(|D|+1)2^{|D|}\leq \\
\sum_{d=1}^n(d+1)2^d(10n)^{-d}\frac{n^d}{d!}\leq \sum_{d=1}^{\infty}(d+1)\frac{5^{-d}}{d!}\leq 0.47.
\end{multline*}
In light of this and (\ref{T1}), we see from (\ref{6}) that
\begin{equation*}
L:=|A|^{1/2}\left(|A|^{1/2}\sigma - 0.12(mq^{-1})^{1/4}(10n)^{n}\right)\leq 0.47(mq^{-1})^{1/2}(10n)^{2n}=:R. \label{<}
\end{equation*} 
Assume that 
$$|A|> (mq^{-1})^{1/2}(10n)^{2n}.
$$
But $n\geq2$; hence, $|A|>100$ and $\sigma=1-|A|^{-1} > 0.99$. Therefore
\begin{equation*}
L> (0.99-0.12)(mq^{-1})^{1/2}(10n)^{2n}>R, 
\end{equation*} 
a contradiction. This completes the proof.

%\newpage 
%%%%%%%%%%%%%%%%%%%%%%%%%%%%%%%%%%%%%%%%%%%%%%%%%%%%%%%%%%%%%
%%%%%%%%%%%%%%%%%%%%%%%%%%%%%%%%%%%%%%%%%%%%%%%%%%%%%%%%%%%%%

\section{Proof of Theorem \ref{th1.2}} 
\label{Sec3} 

We begin with the following simple lemma.

\begin{lem}\label{lem3.1}
	Let $m=m_1m_2$, where $(m_1,m_2)=1$, and assume that we have a bound $|A_2|\leq g(m_2)$ for all $A_2\subset\Z_{m_2}$ with (\ref{0.1}). Then for any $A\subset\Z_m$ with (\ref{0.1}) we have
	$$|A|\leq m_1g(m_2).
	$$
\end{lem}

\begin{proof}
For any $a_1\in \Z_{m_1}$ set
$$A(a_1)=\{a\in A: a\equiv a_1\pmod{m_1}\}.
% : (a_1,a_2) \in A \} .  % \subset \Z_{m_2}.
$$
%$$A(a_1)=\{ a_2\in \Z_{m_2} : (a_1,a_2) \in A \} .  % \subset \Z_{m_2}.
%$$
Then $A(a_1)\pmod{m_2}$ obeys (\ref{0.1}) and, hence, we have $|A(a_1)|\leq g(m_2)$.  Summing over $a_1$ completes the proof.
\end{proof}

\bigskip
 
The idea of the proof of Theorem \ref{th1.2} is the following. First, we restrict our attention to those $\sqrt{x} < m\le x$ for which 
\begin{itemize}
\item[(i)]  its powerful\footnote{A number $n$ is \emph{powerful} if every prime in its prime factorization occurs with exponent at least 2.} part $P(m)=\prod_{p^{\a}||m, \, \a\geq2}p^{\a}$ is at most $\log x$;
\item[(ii)] $\o(m)\leq 2\log\log x$.
\end{itemize}

Almost all integers $m$ satisfy (i) and (ii).  Indeed,
the number of powerful integers $\le y$ is $O(\sqrt{y})$
and hence by partial summation the number of $m\le x$ failing (i) is at most
\[
\ssum{d>\log x \\ d\text{ is powerful}}
\frac{x}{d} \ll \frac{x}{\sqrt{\log x}}.
\]
The number of $m\le x$ failing (ii) is
$O(x/(\log x)^{\log 4-1})$ by the Hardy-Ramanujan \cite{HR} (see also Theorem 1.5 of \cite{N} or Theorem 3 of \cite{F})
estimate
\[
\# \{m\le x : \omega(m)=k\} \ll \frac{x}{\log x} \cdot \frac{(\log\log x+O(1))^{k-1}}{(k-1)!},
\]
after summing over $k>2\log\log x$ (the first summand dominates) and using Stirling's formula.  Thus,
\be\label{fail-i.ii}
\# \{m\le x : m\text{ fails (i) or fails (ii)}\} \ll
\frac{x}{(\log x)^{1/3}}.
\ee
For $m$ obeying (i) and (ii), let $q_1>q_2>...$ be prime divisors of $m/P(m)$ of the form $4k+3$. Take some $q_{2j-1}$ ($j\geq1$) and set $d_{2j-1}=\prod_{i>2j-1}q_i$. Denote $m_1=P(m)d_{2j-1}$ and $m_2=m/m_1$; then $m_2$ is square-free, $\o_3(m_2)$ is odd, and the least prime divisor of $m_2$ of the form $4k+3$ is $q_{2j-1}$. Now suppose that $A\subset\Z_m$ obeys (\ref{0.1}). Applying Lemma \ref{lem3.1} and Theorem \ref{th1.1}, we get by (i) and (ii)
\begin{equation}\label{3.1}
\begin{split}
|A|&\leq P(m)d_{2j-1} (m_2q_{2j-1}^{-1})^{1/2}(10\o(m_2))^{2\o(m_2)}\\
&=m^{1/2}\left(\frac{d_{2j-1}}{q_{2j-1}}\right)^{1/2}P(m)^{1/2}(10\o(m_2))^{2\o(m_2)}\\
&=m^{1/2}\left(\frac{d_{2j-1}}{q_{2j-1}}\right)^{1/2}\exp(O(\log_2 x\log_3 x)),
\end{split}
\end{equation}
where we use iterated logarithm notation $\log_2x=\log\log x$, $\log_3x=\log\log\log x$, etc. We see that our bound is good if $q_{2j-1}$ is significantly larger than $d_{2j-1}$.  If we can find a prime $q_{2j-1}$ such that

\smallskip

(iii) there exists $q_{2j-1}>x^{\e}$ such that $q_{2j-1}>d_{2j-1}^2$,

\smallskip

\noindent
then we have from (\ref{3.1})
\begin{equation*}\label{3.2}
|A|\leq m^{1/2-\e/4}\exp(O(\log_2 x\log_3 x)) \le m^{1/2-\e/5}
\end{equation*}
for large enough $x$, since $m > \sqrt{x}$ and $\e\geq (\log x)^{-1/2}$.
Thus to prove Theorem \ref{th1.2} it suffices to
prove that (iii) holds for all but $O(c(\e)x)$ numbers $m\leq x$.
 Note that if $\eps>0$ is fixed then (iii)
 fails for a positive proportion of all $m$,
 e.g. those which are $m^\eps$-smooth.
 
For a positive integer $m$ and $y>0$, we set
$$D(m,y)=\prod_{ \substack{q<y, \,\,q^{\a}||m \\ q\equiv3\pmod{4}}} q^{\a}.
$$
Theorem \ref{th1.2} will evidently follow from the next lemma.

\begin{lem}[the condition (iii)]\label{lem3.2}
Let $\e\in[(\log x)^{-1/2},1]$ and $c(\e)=\exp(-(\log\e^{-1})^{1/10})$. Then all but $O(c(\e)x)$ integers $m\leq x$ have a prime divisor $q_{2j-1}>x^{\e}$ with $q_{2j-1}>D(m,q_{2j-1})^2$.
\end{lem}

The same proof gives a similar statement with the exponent $2$
replaced by any fixed constant, but we do not need this here.

We note that in the work \cite{Bo} (see also \cite{HT}, Chapter 1, and the work \cite{E}) the following question (very close to (iii)) was studied. For $m\leq x$ and $y>0$ define 
$$d(m,y):=\max\{d|m: P^+(d) < y \},
$$
where $P^+(d)$ is the largest prime divisor of $d$. It was shown in \cite{Bo} that for  any $u>0$, almost all $m$ have about $\beta(u)\log\log m$ prime divisors $p|m$ with $p^u>d(m,p)$, where $\beta\colon [0,+\infty)\to[0,1]$ is a continious increasing function with $\beta(0)=0$ and $\lim_{u\to\infty}\beta(u)=1$. Nevertheless, we have several extra requirements in (iii) and thus we cannot use this result directly.

%%%%%%%%%%%%%%%%%%%%%%%%%%%%%%%%%%%%%%%%%%%%%%%%%%%%%%%%%%%%%%%%%%
%%%%%%%%%%%%%%%%%%%%%%%%%%%%%%%%%%%%%%%%%%%%%%%%%%%%%%%%%%%%%%%%%%

\section{Proof of Lemma \ref{lem3.2}}
\label{Sec4} 

For an integer $m$ and set $T$ of primes, let 
$\omega(m,T)$ be the number of distinct primes from $T$
which divide $m$.
The proof relies on the fact that if $T_1,...,T_r$ are disjoint subset of primes below $y=x^{o(1)}$ as $x\to \infty$, then $\o(m,T_j)$ behave like independent Poisson random variables with parameters $H_1(T_j)=\sum_{p\in T_j}\frac1p$.
Define the Total Variation Distance $d_{TV}(X,Y)$ between two random variables $X$ and $Y$ taking values in a discrete space $\O$ by
$$d_{TV}(X,Y):=\max_{A\subset \O} |\PR(X\in A)-\PR(Y\in A)|=
\frac12\sum_{\o\in\O}\left|\P(X=\o)-\P(Y=\o)\right|.
$$
We cite \cite[Theorem 1]{F}.  Here
\[
H_2(T) := \sum_{p\in T} \frac{1}{p^2}.
\]

\begin{lem}[Ford \cite{F}]\label{dTV}
Let $2\le y\le x$ and suppose that
 $T_1,\ldots,T_r$ are disjoint nonempty sets of primes in $[2,y]$.
Then
\[
d_{TV} \Big( (\omega(n,T_1),\ldots,\omega(n,T_r)), (Z(T_1),\ldots,Z(T_r)) \Big) \ll  \sum_{j=1}^r \frac{H_2(T_j)}{1+H_1(T_j)} + u^{-u}, \quad u=\frac{\log x}{\log y},
\]
where, for any set $T$ of primes, $Z(T)$ is a Poisson random variable with parameter $H_1(T)$, and
$Z(T_1),\ldots,Z(T_r)$ are independent.
\end{lem}

Now we are ready to prove Lemma \ref{lem3.2}. We may suppose that $\e$ is small enough since otherwise the claim follows by taking the implied constant large enough. 
Set 
\[
\theta=C(\log \e^{-1})^{1/10},
\]
 where $C$ is a fixed, large constant.   We consider the primes of the form $4k+3$ in the interval $(x^{\e},x^{\sqrt{\e}}]$. Set $y_0=x^{\sqrt{\e}}$ and $y_j=y_0^{1/\theta^{j}}$ for $j=1,...,J$, where 
\begin{equation}\label{J} 
J=\max\{ j : \theta^{-j}\geq \e^{1/2}\}=\left\lfloor\frac{\log \e^{-1}}{2\log\theta}\right\rfloor \asymp \frac{\log( \e^{-1})}{\log \log (\eps^{-1})}.
\end{equation}
Further, define 
$$T_j=\{q\equiv3\pmod{4} \text{ prime } : q\in(y_j, y_{j-1}] \} 
\qquad (1\le j\le J).
$$
Then $T_1,...,T_J$, are disjoint subsets of primes in $(x^{\e},x^{\sqrt{\e}}]$ and by the Mertens theorem for arithmetic progressions (see \cite{W}), 
\be\label{H1Tj}
\lambda_j := H_1(T_j)=\frac12\log\frac{\log y_{j-1}}{\log y_j}+O(1)=\frac12\log\theta+O(1)\in(\frac13\log\theta,\log\theta),
\ee
 since $\e$ is small enough. 
For a randomly chosen $m\in [1,x]$, 
let $\omega_j=\omega(m,T_j)$ and $Z_j=Z(T_j)$ for each $j$,
and $\boldsymbol{\omega}=(\omega_1,\ldots,\omega_J)$
and $\bfZ=(Z_1,\ldots,Z_J)$.
 Applying Lemma \ref{dTV}, we obtain
\begin{equation*}
d_{TV}(\boldsymbol{\omega},\bfZ) \ll \exp(-\e^{-1/2})+(\log \theta)^{-1}\sum_{q>x^{\e}}q^{-2} \ll \exp(-\e^{-1/2}).
\end{equation*}
Here we used that $\eps \ge (\log x)^{-1/2}$.
In particular, for any event $E\subset\N_0^J$ we have 
\begin{equation}\label{muhaha} 
\left|\P(\boldsymbol{\omega}\in E)-\P(\bfZ\in E)\right| 
\ll \exp(-\e^{-1/2}).
\end{equation} 
Our main idea is to show that the event
\[
E = \{ (e_1,\ldots,e_J)\in \NN_0^J : \exists\, j\le J-3 \text{ with } e_{j+3}=0, e_{j+2}=1, e_{j+1}=0 \text{ and } e_j=1\}
\] 
is very likely.  This corresponds to 
 $\o_{j+3}=0$, $\o_{j+2}=1$, $\o_{j+1}=0$ and $\o_j=1$ for some $j$.
For such $m$, it is then very likely that $q'$, the unique prime
divisor of $m$ in $T_{j+2}$, satisfies
$q' > D(m,q')^2=D(m,y_{j+3})^2$ and that
$q''$,  the unique prime
divisor of $m$ in $T_{j}$, satisfies
$q'' > D(m,q'')^2=D(m,y_{j+1})^2$.
Furthermore, one of the primes $q',q''$ has an odd index,
that is, equals $q_{2h-1}$ for some $h$. 

For any $k\le (J-3)/4$, we have by \eqref{H1Tj}
\begin{align*}
\PR \big( Z_{4k+3}=0, Z_{4k+2}=1, Z_{4k+1}=0, Z_{4k}=1  \big) =
\lambda_{4k+2} \lambda_{4k} \er^{-\lambda_{4k+3}-\lambda_{4k+2}-\lambda_{4k+1}-\lambda_{4k}}
\ge \theta^{-4}.
\end{align*}
The 4-tuples $(Z_{4k+3}, Z_{4k+2}, Z_{4k+1}, Z_{4k})$
are independent for different $k$.
Therefore,
\[
\PR(\bfZ \not\in E) \le (1-\theta^{-4})^{(J/4)-2} \le \er^{-\theta^{-4}(J/4-2)}.
\]
By \eqref{muhaha}, it follows that
\be\label{omeganotinE}
\PR(\boldsymbol{\omega} \not\in E) \ll
\er^{-\theta^{-4}(J/4-2)} + \er^{-\e^{-1/2}}
\ll \er^{-(\log \eps^{-1})^{1/2}} \ll c(\eps)
\ee
using \eqref{J}.

We do not want $m$ to have a big smooth part.  Consider the
condition 

\medskip

(iv) For every $0\le j\le J$, $D(m,y_j) \le y_j^{\theta/2}$.

\medskip

\noindent
Using Theorem 07 of \cite{HT}, it follows  for
some absolute constant $c_0>0$ that the number of $m\le x$
failing (iv) is at most
\begin{align*}
\ll \sum_{j=0}^J x \er^{-c_0 \theta/2} \ll x c(\eps)
\end{align*}
if $C$ is taken large enough in the definition of $\theta$.
Therefore, by \eqref{omeganotinE}, the number of 
integers $m\le x$ which fail (iv) or have $\boldsymbol{\omega}
\not\in E$ is $O(c(\eps) x)$.
By \eqref{fail-i.ii}, the number of $m\le x$ failing (i) or failing (ii)
is likewise $O(c(\eps)x)$.

Consider now an integer $m\le x$ satisfying
(i), (ii), (iv) and $\boldsymbol{\omega}\in E$.
In particular, by (i) the primes $q|m$ with $q>x^{\eps}$
divide $m$ to the first power only.
By  $\boldsymbol{\omega}\in E$, there is at least one $j\le J-3$ such that
\[
\omega_{j+3} = 0, \omega_{j+2}=1, \omega_{j+1}=0, \omega_j=1.
\]
Let  $q'$ be the unique prime
divisor of $m$ in $T_{j+2}$
and $q''$ be  the unique prime
divisor of $m$ in $T_{j}$.
By (iv),
\[
D(m,q') = D(m,y_{j+3}) \le y_{j+3}^{\theta/2} = y_{j+2}^{1/2}
\le (q')^{1/2}.
\]
and likewise
\[
D(m,q'') = D(m,y_{j+1}) \le y_{j+1}^{\theta/2} = y_{j}^{1/2}
\le (q'')^{1/2}.
\]
Furthermore, one of the primes $q',q''$ has an odd index,
that is, equals $q_{2h-1}$ for some $h$. 
This completes the proof of Lemma \ref{lem3.2}.

\section{Proofs of Theorems \ref{th1.4} and \ref{th1.5}}
\label{Sec5}

\begin{proof}[Proof of Theorem \ref{th1.4}]
We may assume that $\e\in(0,1/2)$.  We claim that $M_\e$ contains every number $m$
that has a prime factor $q\ge m^{1-\eps}$ with $q\equiv 3\pmod{4}$.
 To see this, suppose that $A\subset\Z_m$ obeys $(A-A)\cap R_m =\{0\}$. Then by Lemma \ref{lem3.1} (with $m_1=mq^{-1}$, $m_2=q$ and $g(m_2)=1$) we see that $|A|\leq mq^{-1}\leq m^{\e}$.
As each number $m$ has at most one such prime factor $q$,
the number of such $m\le x$ is at least
\[
\sum_{x^{1-\e}\leq q\leq x}\#\{m\in(1,x]: q|m \}
=\frac{x}{2}\log\pfrac{1}{1-\e} + O\left(\frac{x}{\log x}\right),
\]
by the Mertens theorem for arithmetic progressions.
\end{proof}

To prove Theorem \ref{th1.5}, we first need the following two results. 

\begin{lem}[Cohen \cite{C}]\label{lem5.1}
For any prime $p\equiv1\pmod{4}$ there exists $A\subset\Z_p$ obeying (\ref{0.1}) with $|A|\geq \frac{1}{2\log2}\log p$.
\end{lem}

\smallskip 

For sake of completeness, we provide a proof here. We follow the short argument from \cite{G}.

\begin{proof}
Consider the complete graph $G=(V,E)$ with $V=\Z_p$ and the partition $E=E_1\bigsqcup E_2$, where $E_1=\{(x,y) : x-y \,\, \mbox{is a residue}\}$ and $E_2=E\setminus E_1$. Then, by Ramsey's theorem for two colours (see, for instance, \cite{TV}, Theorem 6.9), one can find a complete monochromatic subgraph $G'=(V',E')$ of our graph $G$ with $|V'|=n$ whenever $|V|=p \geq \binom{2n-2}{n-1}$. We thus can find such a subgraph of size $n\geq \log p/\log4$. If $E\subseteq E_2$, then the set $V'$ of its vertices gives an example we need; if $E\subseteq E_1$, then for any non-residue $\xi\in\Z_p$ we get such an example in the form $\xi V'$. The claim follows.
\end{proof}

\begin{lem}\label{lem5.2}
	Let $q_1>q_2$ be primes $3\pmod{4}$. Then there exists $A\subset\Z_{q_1q_2}$ obeying (\ref{0.1}) with $|A|\geq \frac{1}{\log2}\log q_2$. 
\end{lem}

\begin{proof} 
For $j=1,2$, let $V_j=(\Z_{q_j},E_{q_j})$ be the tournament of quadratic residues modulo $q_j$, that is, the directed graph with the set of vertices $\Z_{q_j}$ and $\{a,b\}\in E_{q_j}$ iff $a-b\in R_{q_j}\setminus\{0\}$. It is well-known (\cite{St}; see also \cite{EM}) that any tournament on $n$ vertices contains a transitive subtournament of size $\lfloor\log n/\log2\rfloor+1$ (it follows from the fact that any tournament on $2^n$ vertices contains a transitive subtournament of size $n+1$, which can be proved by an easy induction). Applying this to $V_j$, we find two sets $A_1=\{a^{(1)}_1,...,a^{(1)}_k \}\subset \Z_{q_1}$ and $A_2=\{a^{(2)}_1,...,a^{(2)}_k \}\subset \Z_{q_2}$ of size $k\geq \log q_2/\log 2$ such that $a_s^{(j)}-a_t^{(j)}\in R_{q_j}$ for $1\leq s<t\leq k$ and $j=1,2$.	Then the set $A=\{(a^{(1)}_s, a^{(2)}_{k+1-s}) \}_{s=1}^k\subset\Z_{q_1q_2}$ is what we need, since $\left(\frac{-1}{q_j}\right)=-1$ for $j=1,2$. 
\end{proof}

Note that Graham and Ringrose \cite{GR} showed that the least quadratic non-residue $n(p)$ satisfies $n(p)\gg \log p \log\log\log p$ infinitely often, and one can expect that the same bound holds infinitely often for primes $p\equiv1\pmod{4}$ and $q\equiv3\pmod{4}$ separately. Then, as was mentioned in the introduction, we can use the sets $\xi\cdot\{1,...,n(p) \}$, where $\xi\in\Z_p\setminus R_p$ is any non-residue, instead of those we constructed in the proof of Lemma \ref{lem5.1}, and the sets $\{ (s,n+1-s)\}_{s=1}^{n}$, where $n=\min\{n(q_1),n(q_2)\}$ instead of those from Lemma \ref{lem5.2}.  As this bound on $n(p)$ applies only to a very sparse set of primes $p$,
using it would not affect our lower bound in Theorem \ref{th1.5}.

\begin{proof}[Proof of Theorem \ref{th1.5}]
Let $\omega_j(m,t)=\#\{p\le t: p|m,p\equiv j\pmod{4}\}$,  $j\in\{1,3\}$.
 Consider the set of $m\le x$ such that
 \begin{itemize}
 \item[(a)] $p^2|m$ implies that $p>\log x$;
 \item[(b)] $|\o_j(m,t)-0.5\log\log t|<(\log\log x)^{2/3} \quad (3\leq t\leq m, \, j\in\{1,3\})$.
 \end{itemize}
The number of $m$ failing (a) is $O(x/\log x)$.
Almost all integers satisfy (b), and this 
may be derived from Theorem 7.2 in Kubilius \cite{K},
upon taking $f$ to be the strongly additive function with 
$f(p)=1$ if $p\equiv j\pmod{4}$ and $f(p)=0$ otherwise.
Now suppose that $m\in (\sqrt{x},x]$ obeys (a) and (b).
For $p|m$ and $p>\log x$,
$p$ divides $m$ to the first power.
By Lemma \ref{lem5.1}, for such $p\equiv 1\pmod{4}$ there is a set $A_p\subset \Z_{p}$ with (\ref{0.1}) of size $\gg\log p$. Let $q_1>q_2>...$ be the primes $3\pmod{4}$ dividing $m$, greater than $\log x$; each of them also divides $m$ to the first power. By Lemma \ref{lem5.2}, there is a set $A_{q_{2j-1}q_{2j}}\subset\Z_{q_{2j-1}q_{2j}}$ with (\ref{0.1}) of size $\gg\log q_{2j}$. Further, it is easy to see that the set 
$$
A=\prod_{\substack{p|m\\ p\equiv1\pmod{4} \\ p>\log x}}A_p \times \prod_{\substack{q_{2j}|m\\ q_{2j}>\log x}}A_{q_{2j-1}q_{2j}}\subset \prod_{p|m}\Z_{p} \subseteq \Z_m 
$$         
obeys (\ref{0.1}). It remains to estimate $|A|$. By (b) with $t=m$, we get
\begin{align}\label{5.1}
\log|A|\geq \sum_{\substack{p|m, p>\log x\\ p\equiv1\pmod{4}}}(\log\log p - O(1)) + \sum_{\substack{q_{2j}|m, \\q_{2j}>\log x}}(\log\log q_{2j} - O(1)) \nonumber\\
= \sum_{\substack{p|m, p>\log x\\ p\equiv1\pmod{4}}}\log\log p + 0.5\sum_{\substack{q|m, q>\log x\\q\equiv3\pmod{4}}}\log\log q - O(\log\log x).
\end{align}
Using  Abel's summation technique, we find by (b) that
\begin{align*}
\sum_{\substack{p|m, \log x<p\le \sqrt{x}\\ p\equiv1\pmod{4}}}\log\log p
 &\ge  \o_1(m,\sqrt{x})\log\log \sqrt{x}-\o_1(m,\log x)\log\log\log x - \int_{\log x}^{\sqrt{x}}\frac{\o_1(m,u)du}{u\log u}\\
&=0.25 (\log\log x)^2 + O((\log\log x)^{5/3}).
\end{align*}
Analogously,
\begin{align*}
\sum_{\substack{q|m, \log x<q\le \sqrt{x}\\ q\equiv3\pmod{4}}}\log\log q
&\ge  0.25 (\log\log x)^2 + O((\log\log x)^{5/3}).
\end{align*}
The claim follows from (\ref{5.1}).
\end{proof}

\medskip

\textbf{Acknowledgements.}
The authors thank Sergei Konyagin for suggesting Lemma \ref{lem5.2}.
The first author is supported by National Science Foundation
Grant DMS-1802139. The work of the second author was performed at the Steklov International Mathematical Centre and supported by the Ministry of Science and Higher Education of the Russian Federation (agreement no. 075-15-2019-1614).

\end{document}